\documentclass{amsart}
\usepackage{amsmath}
\usepackage{amssymb}
\usepackage{amsthm}
\usepackage{amsfonts}
\usepackage{mathtools}
\usepackage{fullpage}
\usepackage{booktabs}
\usepackage{enumerate}
\usepackage{graphicx}
\usepackage{xcolor}
\usepackage[shortlabels]{enumitem}
\usepackage{esint}
\usepackage{tikz-cd}
\usepackage{float}
\usepackage{stmaryrd}
\usepackage{url}
\usepackage{breakurl}
\usepackage{hyperref}
\usepackage{fontawesome}
\usepackage{biblatex}

\hypersetup{
	breaklinks=true,   
	colorlinks=true,   
	pdfusetitle=true,  
}

\newtheorem{theorem}{Theorem}
\newtheorem{proposition}{Proposition}

\newtheorem{corollary}{Corollary}
\newtheorem{lemma}{Lemma}
\newtheorem{remark}{Remark}
\newtheorem*{theorem*}{Theorem}
\newtheorem*{proposition*}{Proposition}
\newtheorem*{corollary*}{Corollary}
\newtheorem*{task*}{Task}
\newtheorem*{wish*}{Wish}

\theoremstyle{definition}

\newtheorem{definition}{Definition}
\newtheorem*{solution*}{Solution}
\newtheorem*{idea*}{Idea}

\bibliography{refs.bib}

\title{Approximation of L-functions associated to Hecke cusp eigenforms}

\author[A. Huang]{An Huang}
\address{Department of Mathematics, Brandeis University,	415 South Street, Waltham, MA 02453, USA}
\email{anhuang@brandeis.edu}

\author[K. Spinelli]{Kamryn Spinelli}
\address{Department of Mathematics, Brandeis University,	415 South Street, Waltham, MA 02453, USA}
\email{kspinelli@brandeis.edu}

\begin{document}
	
\maketitle

\begin{abstract}
	We derive a family of approximations for L-functions of Hecke cusp eigenforms,
	according to a recipe first described by Matiyasevich for the Riemann xi
	function. We show that these approximations converge to the true L-function
	and point out the role of an equidistributional notion in ensuring the
	approximation is well-defined, and along the way we demonstrate error formulas
	which may be used to investigate analytic properties of the L-function
	and its derivatives, such as the locations and orders of zeros. Together
	with the Euler product expansion of the L-function, the family of approximations
	also encodes some of the key features of the L-function such as its functional
	equation. As an example, we apply this method to the L-function of the
	modular discriminant and demonstrate that the approximation successfully
	locates zeros of the L-function on the critical line. Finally, we derive
	via Mellin transforms a convolution-type formula which leads to precise
	error bounds in terms of the incomplete gamma function. This formula can
	be interpreted as an alternative definition for the approximation and sheds
	light on Matiyasevich's procedure.
\end{abstract}

\section{Introduction}
\label{sec1}

Given an (often number-theoretic or algebro-geometric) object such as a
number field, Hecke character, Galois representation, algebraic variety,
or modular form, one can often associate to it an L-function: this is a
holomorphic or meromorphic function defined in a right half-plane by a
Dirichlet series $\sum _{n=1}^{\infty }a_{n} n^{-s}$, and extended to the
rest of the complex plane by analytic continuation, where the
$a_{n}$ encode some numerical data about the original object (e.g. values
of a character, Fourier coefficients of a modular form, or point counts
of a variety over finite fields). The prototypical example is the Riemann
zeta function $\zeta (s) = \sum _{n=1}^{\infty }n^{-s}$.

There is broad interest in statements such as the BSD conjecture and the
Bloch-Kato conjecture which relate the values of an L-function
$\Lambda $ and its derivatives at certain special values to properties
of the object it came from. For this reason, one frequently needs to understand
the analytic properties of $\Lambda $. However, these special values often
lie in the half-plane where the Euler product expansion does not hold,
and therefore accessing this analytic data can be difficult. In this article,
we develop an approximation technique based on truncated Euler products
which could be used to investigate these properties. This technique first
appeared in \cite{Matiyasevich2017}, where it was applied to construct
approximations for the Riemann xi function, and was further analyzed in
\cite{NASTASESCU2022126344}, leading to precise error bounds for the approximation
in this case. Subsequently, the same procedure was applied to construct
approximations for Dirichlet L-functions \cite{alzergani2023family}, L-functions
of elliptic curves \cite{nastasescu2023visual}, and generating functions
$\zeta (s) \zeta (a-s)$ of divisor sums \cite{nastasescu2023riemann}.

The main contributions of the present article to the understanding of this
method are the following:
\begin{enumerate}
	\item The approximation technique, based on a regularization and symmetrization
	procedure, applies to L-functions of Hecke cusp eigenforms (Theorem \ref{thm:approximations-convergence}). We show an example which suggests
	that this could be a useful tool for investigating the analytic and number-theoretic
	data of these L-functions, such as the locations and orders of their zeros.
	\item The construction of the approximation technique, in particular the
	convergence of a certain sum over poles of the Euler factors, relies on
	an equidistributional aspect of a sequence involving logarithms of primes
	(Lemma \ref{lem:laurent-coefficients-bounded-growth}). As far as we know,
	the role of equidistribution in the approximation procedure has not been
	described elsewhere.
	\item We interpret the approximation through the lens of Mellin inversion
	and show that the approximation can be equivalently defined as a convolution-type
	integral involving a subseries of the Fourier expansion of the Hecke eigenform
	(Proposition \ref{prop:mellin-transform-approximation}). Furthermore, this
	definition leads to an error formula involving the incomplete gamma function
	(Corollary \ref{cor:error-bound-incomplete-gamma-function}) and explains
	why error formulas of this type have appeared in previous articles utilizing
	this approximation technique in other settings.
\end{enumerate}

We now briefly summarize the structure and main results of this article.
We will apply the approximation technique to L-functions of Hecke cusp
eigenforms; these are modular forms which are simultaneous eigenfunctions
for all the Hecke operators and whose Fourier expansions have zero constant
term. The Fourier coefficients of Hecke eigenforms are of special interest
for their number-theoretic properties. They are multiplicative in the number-theoretic
sense, as was observed by Ramanujan for the function $\tau (n)$, and satisfy
a bounded growth condition. We will recall the necessary ingredients for
the approximation in Section~\ref{sec:structure-L-factors} and the reader
is referred to
\cite{koblitz,apostol,ribet-petersson-ramanujan-conjecture} for more
background on Hecke eigenforms.

More precisely, let $\Gamma = \operatorname{SL}_{2}(\mathbb{Z})$,
$\Gamma _{0}(C)$, or $\Gamma _{1}(C)$ (see Section~\ref{sec:structure-L-factors} for the definition of these congruence subgroups),
and let $f$ be a modular form of weight $k$ for the group $\Gamma $. To
this $f$ is associated a completed L-function defined by the Euler product
$\Lambda (s) = g(s) \prod _{p=2}^{\infty }L_{p}(s)$ in a right half-plane
and extended to the rest of $\mathbb{C}$ by analytic continuation, which
is entire and satisfies $\Lambda (s) = \pm \Lambda (k-s)$. Here,
$L_{p}(s)$ and $g(s)$ are the local L-factors at each finite place and
at infinity, respectively, and each finite place factor $L_{p}(s)$ can
be written as the reciprocal of a quadratic polynomial in $p^{-s}$. To
construct the approximation, we start with a truncation
$\Lambda _{N}^{Euler}(s) = g(s) \prod _{p=2}^{p_{N}} L_{p}(s)$ of the Euler
product. Then we regularize the truncated Euler product by defining
$\Lambda _{N}^{ingoing}(s) = \Lambda _{N}^{Euler}(s) - \Lambda _{N}^{pp}(s)$,
where $\Lambda _{N}^{pp}(s)$ is defined as the sum of the principal parts
of the Laurent expansion of $\Lambda _{N}^{Euler}(s)$ at each of its (infinitely
many) poles. Along the way, in Lemma \ref{lem:laurent-coefficients-bounded-growth}, we make the novel observation
that the convergence of the infinite sum defining
$\Lambda _{N}^{pp}$ is closely related to the equidistribution of a sequence
involving ratios of logarithms of primes. The final step is to impose a
functional equation on the approximation, by setting
$\Lambda _{N}(s) = \Lambda _{N}^{ingoing}(s) \pm \Lambda _{N}^{ingoing}(k-s)$
according to the sign in the functional equation of $\Lambda $. The resulting
function $\Lambda _{N}$ is entire, satisfies the same functional equation
as $\Lambda $, and contains the data of the first $N$ Euler factors of
the original L-function. Since $\Lambda _{N}^{pp}$ is defined by an infinite
sum, in Section~\ref{sec:approximation-construction} we carefully verify
its convergence and bound its magnitude on certain contours in the complex
plane which avoid the poles of $\Lambda _{N}^{Euler}$, called sparse contours.

We employ the bounds of $\Lambda _{N}^{pp}$ on sparse contours in the following
Section~\ref{sec:integral-formulas} to represent $\Lambda _{N}$ as a contour
integral over a vertical line in $\mathbb{C}$. This integral representation
relies on two key facts: first, the poles of $\Lambda _{N}^{Euler}$ are
contained in a left half-plane; and second, the contribution of the principal
part $\Lambda _{N}^{pp}$ to the integrals diminishes to zero as one takes
larger sparse contours. We also derive an analogous integral formula for
$\Lambda $ using just its functional equation and decay properties, and
together the integral formulas lead to an error formula for the approximations:
\begin{theorem}[$=$ Theorem \ref{thm:error-formula}]
	Let $s_{0} \in \mathbb{C}$ be arbitrary. For
	$\sigma > \max (\operatorname{Re}(s_{0}), \operatorname{Re}(k - s_{0}),
	\frac{k+1}{2})$,
	\begin{equation*}
		\Lambda (s_{0}) - \Lambda _{N}(s_{0}) = \int _{\operatorname{Re}(s) =
			\sigma} (\Lambda (s) - \Lambda _{N}^{Euler}(s)) \left (
		\frac{1}{s-s_{0}} \pm \frac{1}{s-k+s_{0}} \right )\ ds.
	\end{equation*}
\end{theorem}
We then prove via a dominated convergence argument in Section~\ref{sec:convergence-of-approximations} the convergence of the approximations:
\begin{theorem}[$=$ Theorem \ref{thm:approximations-convergence}]
	For any $s_{0} \in \mathbb{C}$, $\Lambda _{N}(s_{0})$ converges to
	$\Lambda (s_{0})$ as $N \to \infty $.
\end{theorem}
As an easy consequence, we also observe that differentiating under the
integral sign yields analogous identities and convergence formulas for
derivatives of the L-functions.

\begin{remark}
	\label{rem1}
	The family of approximations can be thought of as an alternative characterization
	of the L-function of $f$ from which some of the fundamental properties
	are apparent. More precisely, for any $N$, the procedure described above
	defines $\Lambda _{N}$ as an explicit entire function without using analytic
	continuation, and then Theorem \ref{thm:approximations-convergence} implies
	that one can define
	$\Lambda (s) = \lim _{N \to \infty} \Lambda _{N}(s)$. From this definition,
	the functional equation of $\Lambda $ is apparent, because
	$\Lambda _{N}$ satisfies the functional equation by its construction. Additionally,
	the integral in Theorem \ref{thm:error-formula} is convergent for any
	$s_{0}$ and therefore this shows that $\Lambda $ is an entire function.
\end{remark}

In Section~\ref{sec:example-modular-discriminant}, we apply the approximation
method to the L-function of the modular discriminant, whose coefficients
are the well-known Ramanujan tau sequence. We observe that the data of
just the first three Euler factors approximates the true L-function closely
on a large region, and successfully predicts the locations of the first
several zeros of the L-function on the critical line with high accuracy.
This serves as evidence that the technique is of number-theoretic interest,
due to its capability to access analytic data such as locations and orders
of zeros of the L-function.

Finally, in Section~\ref{sec:mellin-transforms} we interpret the contour
integral formulas in terms of Mellin transforms and thus derive a convolution-type
formula for the completed L-function as well as the approximations. This
formula for $\Lambda $ is well-known (see e.g.
\cite[pp. 139-141]{koblitz}), and in the case of the approximations
$\Lambda _{N}$, it takes the following form:
\begin{proposition}[$=$ Proposition \ref{prop:mellin-transform-approximation}]
	For all $s \in \mathbb{C}$,
	\begin{equation*}
		\Lambda _{N}(s) = \int _{1}^{\infty }(t^{s-1} \pm t^{k-1-s}) f_{N}
		\left ( \frac{it}{\sqrt{C}} \right )\ dt
	\end{equation*}
	where $f_{N}$ is a certain subseries of the Fourier expansion of $f$.
\end{proposition}
This can be interpreted as an alternative definition for the approximation
L-function which avoids the technical regularization procedure and may
be simpler for numerical computations. Furthermore, it leads to a precise
error bound for the approximation in terms of the incomplete gamma function:
\begin{corollary}[$=$ Corollary \ref{cor:error-bound-incomplete-gamma-function}]
	\begin{equation*}
		\Lambda (s) - \Lambda _{N}(s) = \!\sum _{n=1}^{\infty} c_{n}^{(N)}
		\left ( \left ( \frac{2\pi n}{\sqrt{C}} \right )^{-s} \Gamma \left ( s,
		\frac{2\pi n}{\sqrt{C}} \right ) + (-1)^{P} \left (
		\frac{2\pi n}{\sqrt{C}} \right )^{-k+s} \!\Gamma \left ( k-s,
		\frac{2\pi n}{\sqrt{C}} \right ) \right )
	\end{equation*}
\end{corollary}
Error formulas of this form have appeared previously in other articles
utilizing Matiyasevich's technique
\cite{alzergani2023family,nastasescu2023visual,nastasescu2023riemann}.
This gives weight to the idea of understanding the approximation through
Mellin inversions and convolution-type integrals, and may contribute to
a deeper understanding of its error and convergence properties.

\section{Structure of L-factors}
\label{sec:structure-L-factors}

In this section, we briefly summarize some fundamental facts about the
L-function of a Hecke cusp eigenform and basic observations about the structure
of its Euler factors. Let $\Gamma $ be one of the congruence subgroups
\begin{align*}
	\Gamma _{0}(C) &= \left \{
	\begin{pmatrix}
		a & b
		\\
		c & d
	\end{pmatrix} \in \operatorname{SL}_{2}(\mathbb{Z}) \ \big |\ c
	\equiv 0 \operatorname{mod}C \right \}
	\\
	\Gamma _{1}(C) &= \left \{
	\begin{pmatrix}
		a & b
		\\
		c & d
	\end{pmatrix} \in \operatorname{SL}_{2}(\mathbb{Z}) \ \big |\ c
	\equiv 0 \operatorname{mod}C,\ a \equiv d \equiv 1 \operatorname{mod}C
	\right \}
\end{align*}
for some $C \in \mathbb{Z}_{\geq 1}$ (note that
$\Gamma _{1}(1) = \Gamma _{0}(1) = \operatorname{SL}_{2}(\mathbb{Z})$).
Let $f \in S_{k}(\Gamma , \chi )$ be a Hecke cusp eigenform of weight
$k$ for the group $\Gamma $, with nebentypus
$\chi : (\mathbb{Z}/ C\mathbb{Z})^{\times }\to \mathbb{C}^{\times}$ if
$\Gamma = \Gamma _{0}(C)$ (otherwise we assume $\chi $ is the trivial character).
Denote by $\sum _{n=1}^{\infty} a_{n} q^{n}$ the Fourier expansion of
$f$. The associated L-function of $f$ is defined in a half-plane by the
Dirichlet series $L(s) = \sum _{n=1}^{\infty }a_{n} n^{-s}$, and elsewhere
by analytic continuation; we focus on its completed version
$\Lambda (s) = g(s) L(s)$ where
$g(s) = C^{s/2} (2\pi )^{-s} \Gamma (s)$, which is an entire function.
In this case, there are the following well-known results on the structure
and properties of $\Lambda $ \cite{koblitz,apostol}.
%
\begin{proposition}[Euler factor decomposition]
	\label{prop:euler-factor-decomposition}
	$L$ decomposes as a product of Euler factors:
	$L(s) = \prod _{p} L_{p}(s) = \prod _{p} (1 - a_{p} p^{-s} + \chi (p) p^{k-1}
	(p^{-s})^{2})^{-1}$.
\end{proposition}
%
\begin{proposition}[Functional equation]
	\label{prop:functional-equation}
	$\Lambda $ satisfies the functional equation
	$\Lambda (s) = \break (-1)^{P} \Lambda (k-s)$ for some choice of sign
	$(-1)^{P}$.
\end{proposition}

From Proposition \ref{prop:euler-factor-decomposition}, one can analyze
the singularities of the L-factors. The Euler factor $L_{p}(s)$ can be
written as
$((1 - \alpha _{1,p} p^{-s}) (1 - \alpha _{2,p} p^{-s}))^{-1}$ where
$\alpha _{1,p} + \alpha _{2,p} = a_{p}$ and
$\alpha _{1,p} \alpha _{2,p} = \chi (p) p^{k-1}$, and the
$\alpha _{i,p}$ are constrained by the Petersson-Ramanujan conjecture:
%
\begin{theorem}[Petersson-Ramanujan conjecture \cite{ribet-petersson-ramanujan-conjecture,deligne,apostol}]
	\label{thm:petersson-ramanujan}
	If $\Gamma = \Gamma _{1}(C)$,
	$|\alpha _{1,p}| = |\alpha _{2,p}| = p^{(k-1)/2}$. Furthermore, owing to
	the direct sum decomposition
	$S_{k}(\Gamma _{1}(C)) = \bigoplus _{\chi }  S_{k}(\Gamma _{0}(C),
	\chi )$, the same is true for $\Gamma = \Gamma _{0}(C)$.
\end{theorem}
%
\begin{corollary}
	\label{cor1}
	The Dirichlet series representation of $L$ converges in the half-plane
	$\operatorname{Re}(s) > \frac{k+1}{2}$.
\end{corollary}

The following lemma is a straightforward application of Theorem \ref{thm:petersson-ramanujan} and L'Hopital's Rule.
%
\begin{lemma}
	\label{lem:Lp-poles}
	Let $p$ be a prime, and let $\alpha $ be one of the reciprocal roots of
	$L_{p}(s)^{-1}$. Then $(1 - \alpha p^{-s})^{-1}$ has infinitely many poles
	on the vertical line $\operatorname{Re}(s) = \frac{k-1}{2}$, all equally
	spaced on this line by distance $\frac{2\pi}{\log p}$. The residue at each
	of these poles is $\frac{1}{\log p}$. Therefore $L_{p}(s)$ either has double
	poles on the line $\operatorname{Re}(s) = \frac{k-1}{2}$ spaced by distance
	$\frac{2\pi}{\log p}$, or it has two collections of non-coinciding simple
	poles on the line $\operatorname{Re}(s) = \frac{k-1}{2}$, each spaced by
	$\frac{2\pi}{\log p}$.
\end{lemma}
%
\begin{lemma}
	\label{lem:coinciding-poles}
	Any pair of distinct L-factors has at most four common poles.
\end{lemma}
\begin{proof}
	Let $p \neq q$ be distinct primes. The scenario giving the most coinciding
	poles is when $L_{p}$ and $L_{q}$ each have two subsets of non-coinciding
	poles $D_{p, 1}, D_{p,2}$ and $D_{q,1}, D_{q,2}$ positioned on the vertical
	lines $\operatorname{Re}(s) = \frac{k-1}{2}$ with spacing
	$\frac{2\pi}{\log p}$ and $\frac{2\pi}{\log q}$ respectively. In this scenario,
	each pair of $D_{p,i}$ and $D_{q,j}$ can share at most one common pole
	because $\frac{\log p}{\log q}$ is irrational. As there are four possible
	choices of the\vspace{1.5pt} pair $(i, j)$, the statement follows.
\end{proof}

\section{Approximation of $\Lambda (s)$ by truncated Euler products}
\label{sec:approximation-construction}

We now describe a procedure for constructing an approximation of the completed
L-function $\Lambda $, mimicking the technique in
\cite{Matiyasevich2017,alzergani2023family,nastasescu2023visual,nastasescu2023riemann}.
Pick $N \geq 1$ and consider the collection of primes up to $p_{N}$. The
truncated Euler product is defined by
$\Lambda _{N}^{Euler}(s) = g(s) \prod _{p=2}^{p_{N}} L_{p}(s)$, and it
differs from the full L-function $\Lambda $ in two crucial respects: it
is not entire, and does not satisfy the functional equation. We will proceed
in two steps: first regularizing the truncated Euler product so that it
is entire, and then imposing the functional equation.

Following Matiyasevich's recipe, we begin by subtracting off the singular
part of the Laurent expansion of $\Lambda _{N}^{Euler}$ at each pole; each
of these poles is of finite order because $\Lambda _{N}^{Euler}$ is a product
of finitely many Euler factors, each with either simple or double poles.
Set $D_{N} = \bigcup _{p=2}^{p_{N}} \{\text{poles of $L_{p}$}\}$ (see Lemma \ref{lem:Lp-poles}) and
$S_{N} = D_{N} \cup \{\text{poles of $g$}\} = D_{N} \cup \mathbb{Z}_{
	\leq 0}$. For each $s_{\star }\in S_{N}$, we write the Laurent expansion
of $\Lambda _{N}^{Euler}$ at $s = s_{\star}$ as
$\Lambda _{N}^{Euler}(s) = \sum _{k = -k_{s_{\star}}}^{\infty }\rho _{s_{
		\star}}^{(k)} (s-s_{\star})^{k}$. We then define the principal part of
$\Lambda _{N}^{Euler}$ to be
\begin{equation*}
	\Lambda _{N}^{pp}(s) = \sum _{s_{\star }\in S_{N}} \sum _{k = -k_{s_{
				\star}}}^{-1} \rho _{s_{\star}}^{(k)} (s-s_{\star})^{k},
\end{equation*}
so that
\begin{equation*}
	\Lambda _{N}^{ingoing}(s) = \Lambda _{N}^{Euler}(s) - \Lambda _{N}^{pp}(s)
\end{equation*}
is a well-defined, entire function. The final step of constructing the
approximation is to set
\begin{equation*}
	\Lambda _{N}(s) = \Lambda _{N}^{ingoing}(s) + (-1)^{P} \Lambda _{N}^{ingoing}(k-s)
\end{equation*}
according to the sign of the functional equation for $\Lambda $. The result
is that $\Lambda _{N}$ is an entire function that satisfies the same functional
equation as $\Lambda $ and, informally speaking, carries the information
of the first $N$ Euler factors of $L$.

The main goals of this section are to explain the convergence of
$\Lambda _{N}^{pp}(s)$ -- ensuring that $\Lambda _{N}$ is well-defined
-- and furthermore show that $\Lambda _{N}^{pp}$ satisfies a decay property
on certain well-behaved contours in the complex plane. For readability,
we will omit most of the technical details underlying these arguments and
only briefly summarize the main ideas of the proofs. The interested reader
is encouraged to refer to Appendix~\ref{app:detailed-arguments} for complete proofs.
We begin by bounding the growth of the Laurent coefficients of the L-factors,
which will be crucial for both of our objectives. The following lemma shows
that this is a consequence of the equidistribution of a sequence involving
ratios of logarithms of primes, and its proof is the only one we will include
here in detail due to the novelty of the equidistribution observation.

\begin{lemma}
	\label{lem:laurent-coefficients-bounded-growth}
	Let $p, q$ be primes and let $s_{\star}$ be a pole of
	$\frac{1}{1 - \alpha _{p} p^{-s}}$. Denote by $\rho ^{(k)}$ the $k$-th
	Laurent coefficient of $\frac{1}{1 - \alpha _{q} q^{-s}}$ at
	$s = s_{\star}$. If $s_{\star}$ is also a pole of
	$\frac{1}{1 - \alpha _{q} q^{-s}}$ or if $p = q$, then
	$|\rho ^{(k)}|$ is bounded above by a constant depending only on $q$ and
	$k$. If $p \neq q$ and $s_{\star}$ is not a pole of
	$\frac{1}{1 - \alpha _{q} q^{-s}}$, then $|\rho ^{(k)}|$ is bounded above
	by a constant depending on $q$ and $k$, times $|s_{\star}|^{2k+2}$.
\end{lemma}
\begin{proof}
	If $s_{\star}$ is a pole of $\frac{1}{1 - \alpha _{q} q^{-s}}$, then the
	statement holds because this function is invariant under
	$s \mapsto s + \frac{2\pi n}{\log q} i$, $n \in \mathbb{Z}$, meaning that
	the coefficients of the Laurent expansion about a pole do not depend
	on the pole in question. The situation is similar when $s_{\star}$ is not
	a pole of $\frac{1}{1 - \alpha _{q} q^{-s}}$, but $p = q$: in this case,
	the poles of $\frac{1}{1 - \alpha _{p} p^{-s}}$ and
	$\frac{1}{1 - \alpha _{q} q^{-s}}$ are both equally spaced on the line
	$\operatorname{Re}(s) = \frac{k-1}{2}$ by distance
	$\frac{2\pi}{\log p}$, and both of these expressions are invariant under
	$s \mapsto s + \frac{2\pi n}{\log p} i$. Therefore the Laurent coefficients
	of $\frac{1}{1 - \alpha _{q} q^{-s}}$ at $s_{\star}$ will be the same regardless
	of which pole of $\frac{1}{1 - \alpha _{p} p^{-s}}$ is chosen.
	
	The most delicate case is when $p \neq q$ and $s_{\star}$ is not a pole
	of $\frac{1}{1 - \alpha _{q} q^{-s}}$, and in this case the bound involves
	an equidistributional idea. Taking successive derivatives of
	$\frac{1}{1 - \alpha _{q} q^{-s}}$, the numerators can be easily bounded
	above, so the main point is to prove that
	$|\frac{1}{1 - \alpha _{q} q^{-s_{\star}}}|$ is bounded above by a constant
	times $|s_{\star}|^{2}$, or equivalently that
	$|1 - \alpha _{q} q^{-s_{\star}}| |s_{\star}|^{2}$ is bounded below by
	a constant greater than $0$. Writing
	$\alpha _{q} = q^{\frac{k-1}{2} + I_{q} i}$ and remembering that the poles
	of $\frac{1}{1 - \alpha _{p} p^{-s}}$ are of the form
	$s_{\star }= \frac{k-1}{2} + (I_{p} + \frac{2\pi n}{\log p}) i$ where
	$n \in \mathbb{Z}$, we have
	$1 - \alpha _{q} q^{-s_{\star}} = 1 - q^{(I_{q} - I_{p} -
		\frac{2\pi n}{\log p}) i} = 1 - \exp \left ( ((I_{q} - I_{p}) \log q -
	2\pi n \frac{\log q}{\log p}) i \right )$. This quantity is small whenever
	the exponent is close to a multiple of $2\pi i$, but we will show that
	it can't get too close to such a multiple too fast.
	
	Note that for any $\theta $,
	$|1 - \exp (i\theta )| \geq \{\frac{\theta}{2\pi}\}$ where
	$\{x\} = |x - \lfloor x+\frac{1}{2} \rfloor |$ (in other words, the absolute
	difference between $x$ and the nearest integer), and therefore it suffices
	to prove that
	$\{\frac{(I_{q} - I_{p}) \log q - 2\pi n \frac{\log q}{\log p}}{2\pi}
	\} |s_{\star}|^{2}$ is bounded away from zero. Adjusting by a constant
	depending only on $p$ and $q$, it's equivalent to show that
	$n^{2} \left \{ n \frac{\log q}{\log p} \right \}$ is bounded away from
	zero. It is well-known that
	$\left \{ n \frac{\log q}{\log p} \right \}$ is equidistributed in the
	interval $(0, 1)$ because $\frac{\log q}{\log p}$ is irrational for any
	distinct $p$ and $q$. We claim that if $a_{n}$ is an equidistributed sequence
	in $(0,1)$, then $n^{2} a_{n}$ is bounded away from zero, from which the
	claim follows. Suppose not, so there exists a subsequence
	$n_{k}^{2} a_{n_{k}}$ converging to $0$; then for any $\epsilon > 0$, there
	is a $K$ so that for all $k \geq K$,
	$n_{k}^{2} a_{n_{k}} < \epsilon $ or equivalently
	$a_{n_{k}} < \frac{\epsilon}{n_{k}^{2}}$. Taking $f(x) = x^{-1/2}$ which
	is a Riemann-integrable function on $(0,1)$, this shows that
	$f(a_{n_{k}}) > \frac{n_{k}}{\sqrt{\epsilon}}$, and hence
	\begin{equation*}
		\frac{1}{n_{k}} \sum _{n=1}^{n_{k}} f(a_{n}) \geq \frac{1}{n_{k}} f(a_{n_{k}})
		> \frac{1}{n_{k}} \frac{n_{k}}{\sqrt{\epsilon}} =
		\frac{1}{\sqrt{\epsilon}}.
	\end{equation*}
	But this contradicts the Riemann integral characterization of equidistribution,
	which says that $a_{n}$ is equidstributed if for every Riemann-integrable
	$f$ on $(0,1)$,
	$\lim _{N \to \infty} \frac{1}{N} \* \sum _{n=1}^{N} f(a_{n}) = \int _{0}^{1}
	f(x)\ dx$. This completes the proof.
\end{proof}

In order to analyze the approximation, we will utilize integrals over contours
which sufficiently avoid the poles of $\Lambda _{N}^{Euler}$.
%
\begin{definition}
	\label{defn1}
	We say a closed contour $\mathcal{C}$ in $\mathbb{C}$ is sparse if for
	all $s \in \mathcal{C}$,
	\begin{equation*}
		\min _{s_{\star }\in D_{N}} |s - s_{\star}| \geq a \quad \text{and}
		\quad \min _{s_{\star }\in \mathbb{Z}_{\leq 0}} |s - s_{\star}| \geq b
	\end{equation*}
	for some positive constants $a$ and $b$ depending only on $N$.
\end{definition}
%
\begin{lemma}
	\label{lem4}
	There exist arbitrarily large rectangular contours $\mathcal{C}$ which
	are sparse.
\end{lemma}
\begin{proof}
	Since the poles of $g(s)$ are located at the nonpositive integers, it's
	straightforward to draw the left- and right-hand sides of the rectangle
	at uniform distance $a$ from these poles. For the top and bottom sides,
	since the poles of each $L_{p}$ occur at regular spacing on the vertical
	line $\operatorname{Re}(s) = \frac{k-1}{2}$, the poles cannot occur too
	often on this line, and the constant $b$ can be found by counting how many
	poles may occur in a segment of this line and considering the average gap
	between them.
\end{proof}

With these tools in hand, we will now show that $\Lambda _{N}^{pp}$ is
convergent for all $s \notin S_{N}$ and decays on the order of
$\frac{1}{1 + |s|}$ for $s$ belonging to a sparse contour.

\begin{lemma}
	\label{lem:pp-Lpn-convergent-sparse-contour-bound}
	The function
	$\Lambda _{N, n}^{pp}(s) = \sum _{s_{\star }
		\text{ pole of $L_{p_{n}}(s)$}} \sum _{k = -k_{s_{\star}}}^{-1} \rho _{s_{
			\star}}^{(k)} (s-s_{\star})^{k}$ is convergent for $s$ not a pole of
	$L_{p_{n}}(s)$. Furthermore, there exists a constant $K_{n}$ so that
	$|\Lambda _{N,n}^{pp}(s)| \leq \frac{K_{n}}{1+|s|}$ for all $s$ belonging
	to a sparse contour.
\end{lemma}
\begin{proof}
	Taking the Laurent expansion of $\Lambda _{N}^{Euler}(s)$ about a pole
	$s = s_{\star}$, each Laurent coefficient is bounded by a polynomial function
	of $|s_{\star}|$ thanks to Lemma \ref{lem:laurent-coefficients-bounded-growth}, times the value of
	$\Gamma $ or one of its derivatives at $s_{\star}$. Since the poles of
	$L_{p_{n}}$ are discrete and $\Gamma $ has the Schwartz decay property
	on any vertical line in the complex plane, the sum defining
	$\Lambda _{N, n}^{pp}(s)$ converges, proving the first claim. A technical
	computation involving Laurent expansions of each Euler factor shows that
	$(1 + |s|) \Lambda _{N, n}^{pp}(s)$ is bounded by a sum of
	\begin{equation*}
		S_{k, \ell}(s) = \sum _{s_{\star }\text{ pole of $L_{p_{n}}(s)$}} |
		\rho _{g, s_{\star}}^{(\ell )}| (1 + |s_{\star}|)^{p_{\ell}} |s_{
			\star}|^{k} \frac{1 + |s|}{|s_{\star}|^{k} |s - s_{\star}|^{-k}},
	\end{equation*}
	where $\rho _{g, s_{\star}}^{(\ell )}$ are Laurent coefficients of
	$g(s)$, for finitely many values of $k$ and $\ell $. So for the second
	claim, it suffices to show that $(1 + |s|) S_{k, \ell}(s)$ is bounded by
	a constant independent of $s$. The final factor in the summand is dominated
	by a constant independent of $s$ and $s_{\star}$ whenever $s$ belongs to
	a sparse contour, and the contribution of $\Gamma $ and its derivatives
	to the $\rho _{g, s_{\star}}^{(\ell )}$ ensures that the sum
	$\sum _{s_{\star }\text{ pole of $L_{p_{n}}(s)$}} |\rho _{g, s_{\star}}^{(
		\ell )}| (1 + |s_{\star}|)^{p_{\ell}} | |s_{\star}|^{k}$ converges, so
	the sparse contour bound follows.
\end{proof}

\begin{remark}
	\label{rem2}
	The proof of Lemma \ref{lem:laurent-coefficients-bounded-growth}, which
	is crucial for constructing the approximation, relies on an equidistribution
	property of the poles of the L-factors, which suggests that this property
	is a key ingredient in why the approximation technique works. As a result,
	we anticipate that the technique could be adapted to more general scenarios
	where the poles of the Euler factors satisfy an analogous equidistribution
	property.
\end{remark}

\begin{lemma}
	\label{lem:pp-g-convergent-sparse-contour-bound}
	The function
	$\Lambda _{N,g}^{pp}(s) = \sum _{s_{\star }\text{ pole of g(s)}} \sum _{k
		= -k_{s_{\star}}}^{-1} \rho _{s_{\star}}^{(k)} (s - s_{\star})^{k}$ is
	convergent for $s$ not a pole of $g(s)$. Furthermore, there exists a constant
	$K_{g}$ so that $|\Lambda _{N,g}^{pp}(s)| \leq \frac{K_{g}}{1+|s|}$ for
	all $s$ belonging to a sparse contour.
\end{lemma}
\begin{proof}
	$\Lambda _{N}^{Euler}$ has simple poles at the nonpositive integers, and
	the residue at $s = -n$ is bounded by $\frac{M c^{n}}{n!}$ for some constant
	$c$. So
	\begin{equation*}
		|\Lambda _{N,g}^{pp}(s)| \leq M \sum _{n = 0}^{\infty }
		\frac{c^{n}}{n!} |s+n|^{-1}.
	\end{equation*}
	Therefore, the $n!$ term in the denominator dominates, ensuring convergence
	of the sum. For the second claim, we write
	\begin{equation*}
		(1 + |s|) |\Lambda _{N,g}^{pp}(s)| \leq M \sum _{n = 0}^{\infty }
		\frac{c^{n}}{(n-1)!} \frac{1+|s|}{n |s+n|}.
	\end{equation*}
	The last factor is bounded by a constant independent of $n$ and $s$ whenever
	$s$ belongs to a sparse contour, and the sum
	$\sum _{n=0}^{\infty }\frac{c^{n}}{(n-1)!}$ converges, proving the second
	claim.
\end{proof}

\begin{proposition}
	\label{prop:lambda-pp-convergent-sparse-contour-bound}
	The function $\Lambda _{N}^{pp}(s)$ is convergent for
	$s \notin S_{N}$. For $s \in \mathcal{C}$ a sparse contour,
	$|\Lambda _{N}^{pp}(s)| \leq \frac{K}{1 + |s|}$ where $K$ is a constant
	depending only on the $L$-function and $N$.
\end{proposition}
\begin{proof}
	Due to Lemma \ref{lem:coinciding-poles}, only finitely many poles of the
	$L_{p}(s)$ factors can coincide. Therefore
	\begin{align*}
		\Lambda _{N}^{pp}(s) &= \sum _{n=1}^{N} \Lambda _{N,n}^{pp}(s) +
		\Lambda _{N,g}^{pp}(s)
		\\
		&\quad - (
		\text{finitely many terms due to coinciding poles of different
			$L_{p}(s)$ factors}),
	\end{align*}
	and each of these pieces is convergent thanks to  Lemmas \ref{lem:pp-Lpn-convergent-sparse-contour-bound} and \ref{lem:pp-g-convergent-sparse-contour-bound}, proving the first assertion.
	For the second statement, Lemmas \ref{lem:pp-Lpn-convergent-sparse-contour-bound} and \ref{lem:pp-g-convergent-sparse-contour-bound} provide bounds for
	$\Lambda _{N,n}^{pp}(s)$ and $\Lambda _{N,g}^{pp}(s)$ of the necessary
	form. Each of the terms coming from coinciding poles has the form
	$\rho (s - s_{\star})^{k}$ for some $k \leq -1$, so recalling that
	$s$ belongs to a sparse contour, one can produce individual bounds of the
	required form for each of these terms, which completes the proof.
\end{proof}

\section{Integral formulas for $\Lambda $ and $\Lambda _{N}$}
\label{sec:integral-formulas}

In this section, we deduce integral representations of the completed L-function
$\Lambda $ and its approximation $\Lambda _{N}$ which will be used later
to analyze the convergence of the approximation. The main ingredients in
the integral formula for $\Lambda $ are the functional equation and decay
properties of the completed L-function.
%
\begin{lemma}
	\label{lem:integral-formula-lambda}
	Let $s_{0} \in \mathbb{C}$ be arbitrary. Let
	$\sigma > \max (\operatorname{Re}(s_{0}), \operatorname{Re}(k-s_{0}))$.
	Then
	\begin{equation*}
		\Lambda (s_{0}) = \frac{1}{2\pi i} \int _{\operatorname{Re}(s) =
			\sigma} \Lambda (s) \left ( \frac{1}{s-s_{0}} + (-1)^{P}
		\frac{1}{s-k+s_{0}} \right )\ ds
	\end{equation*}
	where $(-1)^{P}$ is the sign in the functional equation of
	$\Lambda $.
\end{lemma}
\begin{proof}
	Let $\mathcal{C}$ be a rectangular contour defined by
	$\operatorname{Re}(s) = \sigma $,
	$\operatorname{Re}(s) = k - \sigma $, $\operatorname{Im}(s) = \tau $, and
	$\operatorname{Im}(s) = -\tau $, with $\tau $ large enough that
	$\mathcal{C}$ encloses $s_{0}$ and $k-s_{0}$. This contour is invariant
	under the mapping $s \mapsto k-s$. By the functional equation of
	$\Lambda $,
	\begin{equation*}
		\Lambda (s_{0}) = \frac{\Lambda (s_{0})}{2} + (-1)^{P}
		\frac{\Lambda (k-s_{0})}{2} = \frac{1}{2\pi i} \int _{\mathcal{C}}
		\frac{\Lambda (s)}{2} \left ( \frac{1}{s-s_{0}} + (-1)^{P}
		\frac{1}{s-k+s_{0}} \right )\ ds.
	\end{equation*}
	Because $\Lambda $ satisfies the functional equation
	$\Lambda (s) = (-1)^{P} \Lambda (k-s)$, this integral is invariant under
	the change of variables $s \mapsto k-s$. Setting $\mathcal{C}'$ to be the
	top and right sides of the rectangle $\mathcal{C}$, we get
	\begin{equation*}
		\Lambda (s_{0}) = \frac{1}{2\pi i} \int _{\mathcal{C}'} \Lambda (s)
		\left ( \frac{1}{s-s_{0}} + (-1)^{P} \frac{1}{s-k+s_{0}} \right )\ ds.
	\end{equation*}
	Taking the limit as $\tau \to +\infty $, the integral over the top side
	of the triangle tends to zero (one can see this by recognizing
	$\Gamma (s) L(s)$ as the Mellin transform of
	$\sum _{n=1}^{\infty} a_{n} \exp (-nz)$). Thus, only the contribution of
	the right-hand segment in $\mathcal{C}'$ appears in the limit, yielding
	the claim:
	\begin{equation*}
		\Lambda (s_{0}) = \frac{1}{2\pi i} \int _{\operatorname{Re}(s) =
			\sigma} \Lambda (s) \left ( \frac{1}{s-s_{0}} + (-1)^{P}
		\frac{1}{s-k+s_{0}} \right )\ ds.\qedhere
	\end{equation*}
\end{proof}
%
\begin{corollary}
	\label{cor:integral-formula-lambda-coefficients}
	Let $s_{0} \in \mathbb{C}$ be arbitrary. Let
	$\sigma > \max (\operatorname{Re}(s_{0}), \operatorname{Re}(k-s_{0}),
	\frac{k+1}{2})$. Then
	\begin{equation*}
		\Lambda (s_{0}) = \frac{1}{2\pi i} \int _{\operatorname{Re}(s) =
			\sigma} g(s) \sum _{n=1}^{\infty }\frac{a_{n}}{n^{s}} \left (
		\frac{1}{s-s_{0}} + (-1)^{P} \frac{1}{s-k+s_{0}} \right )\ ds.
	\end{equation*}
\end{corollary}
\begin{proof}
	This follows directly from Lemma \ref{lem:integral-formula-lambda}: by
	definition
	$\Lambda (s) = g(s) \sum _{n=1}^{\infty }\frac{a_{n}}{n^{s}}$ for
	$\operatorname{Re}(s) > \frac{k+1}{2}$.
\end{proof}

Next we deduce a similar formula for $\Lambda _{N}$. Similarly to before,
the functional equation and decay properties play a key role along with
the bound of $\Lambda _{N}^{pp}$ on sparse contours.
%
\begin{lemma}
	\label{lem:integral-formula-approximation}
	Let $s_{0} \in \mathbb{C}$ be arbitrary. Let
	$\sigma > \max (\operatorname{Re}(s_{0}), \operatorname{Re}(k-s_{0}),
	\frac{k+1}{2})$. Then
	\begin{equation*}
		\Lambda _{N}(s_{0}) = \frac{1}{2\pi i} \int _{\operatorname{Re}(s) =
			\sigma} \Lambda _{N}^{Euler}(s) \left ( \frac{1}{s-s_{0}} + (-1)^{P}
		\frac{1}{s-k+s_{0}} \right )\ ds
	\end{equation*}
	where $(-1)^{P}$ is the sign in the functional equation of
	$\Lambda $.
\end{lemma}
\begin{proof}
	Let $\mathcal{C}$ be a rectangular sparse contour. Tracing the construction
	of the approximation, by Cauchy's integral formula we have
	\begin{align*}
		\Lambda _{N}(s_{0}) &= \Lambda _{N}^{ingoing}(s_{0}) + (-1)^{P}
		\Lambda _{N}^{ingoing}(k-s_{0})
		\\
		&= \frac{1}{2\pi i} \int _{\mathcal{C}} \Lambda _{N}^{ingoing}(s)
		\left ( \frac{1}{s-s_{0}} + (-1)^{P} \frac{1}{s-k+s_{0}} \right )\ ds
		\\
		&= \frac{1}{2\pi i} \int _{\mathcal{C}} \Lambda _{N}^{Euler}(s)
		\left ( \frac{1}{s-s_{0}} + (-1)^{P} \frac{1}{s-k+s_{0}} \right )\ ds
		\\
		&\quad - \frac{1}{2\pi i} \int _{\mathcal{C}} \Lambda _{N}^{pp}(s)
		\left ( \frac{1}{s-s_{0}} + (-1)^{P} \frac{1}{s-k+s_{0}} \right )\ ds.
	\end{align*}
	For $|s|$ large, Proposition \ref{prop:lambda-pp-convergent-sparse-contour-bound} implies
	$\left | \Lambda _{N}^{pp}(s) \frac{1}{s-s_{0}} \right | \leq
	\frac{2K}{|s|^{2}}$ and
	$\left | \Lambda _{N}^{pp}(s) \frac{1}{s-k+s_{0}} \right | \leq
	\frac{2K}{|s|^{2}}$. So as the sparse contours $\mathcal{C}$ get larger,
	the integral term involving $\Lambda _{N}^{pp}(s)$ is bounded in magnitude
	by
	$\frac{4K}{2\pi}
	\frac{\operatorname{length}(\mathcal{C})}{\min _{s \in \mathcal{C}} |s|^{2}}$.
	This quantity tends to zero in the limit as the interiors of the contours
	$\mathcal{C}$ cover the complex plane provided that
	$
	\frac{\operatorname{length}(\mathcal{C})}{\min _{s \in \mathcal{C}} |s|^{2}}$
	tends to zero, which can be guaranteed, and therefore the integral involving
	$\Lambda _{N}^{pp}$ has zero contribution in the limit. For the integral
	involving $\Lambda _{N}^{Euler}$, define two new closed contours
	$\mathcal{C}_{L}$ and $\mathcal{C}_{R}$ by the parts of
	$\mathcal{C}$ to the left and right of the line
	$\operatorname{Re}(s) = \sigma $ respectively, together with the segment
	of $\operatorname{Re}(s) = \sigma $ which completes each to a closed contour.
	Since $\mathcal{C}_{R}$ contains no poles of
	$\Lambda _{N}^{Euler}(s)$ and does not enclose $s_{0}$ or $k-s_{0}$,
	\begin{equation*}
		\frac{1}{2\pi i} \int _{\mathcal{C}_{R}} \Lambda _{N}^{Euler}(s)
		\left ( \frac{1}{s-s_{0}} + (-1)^{P} \frac{1}{s-k+s_{0}} \right )\ ds =
		0
	\end{equation*}
	implying that
	\[\frac{1}{2\pi i} \int _{\mathcal{C}} \Lambda _{N}^{Euler}(s) \left (
	\frac{1}{s-s_{0}} + (-1)^{P} \frac{1}{s-k+s_{0}} \right )\ ds = \frac{1}{2\pi i} \int _{\mathcal{C}_{L}} \Lambda _{N}^{Euler}(s)
	\left ( \frac{1}{s-s_{0}} + (-1)^{P} \frac{1}{s-k+s_{0}} \right )\ ds.\]
	
	We will show that only one side of the rectangle $\mathcal{C}_{L}$ matters
	in the limit as we take larger sparse contours. Consider the magnitude
	of
	\begin{equation*}
		|\Lambda _{N}^{Euler}(s)| = |g(s)| \prod _{p=2}^{p_{N}} |L_{p}(s)| =
		\left ( \frac{\sqrt{C}}{2\pi} \right )^{\operatorname{Re}(s)} |
		\Gamma (s)| \prod _{p=2}^{p_{N}} |L_{p}(s)|,
	\end{equation*}
	on the left-hand segment of $\mathcal{C}_{L}$ as the contours
	$\mathcal{C}$ get larger; in other words,
	$\lim _{n \to \infty} |\Lambda _{N}^{Euler}(-n-\frac{1}{2} + i\tau )|$.
	As $n \to \infty $, each $|L_{p}(-n-\frac{1}{2} + i\tau )| \to 0$, and
	$|\Gamma (-n-\frac{1}{2} + i\tau )| \leq |\Gamma (-n-\frac{1}{2})| =
	\frac{|\Gamma (\frac{1}{2})|}{|(-n-\frac{1}{2}) \cdots (-\frac{1}{2})|}
	\leq \frac{2 |\Gamma (\frac{1}{2})|}{n!}$ overpowers the factor
	$\left ( \frac{\sqrt{C}}{2\pi} \right )^{\operatorname{Re}(s)} =
	\left ( \frac{2\pi}{\sqrt{C}} \right )^{n+\frac{1}{2}}$. So overall, as
	$\mathcal{C}$ gets larger, the contribution to the integral of the left-hand
	segment of $\mathcal{C}_{L}$ tends to zero.
	
	Next, consider the top horizontal segment of $\mathcal{C}_{L}$. For any
	fixed $\operatorname{Re}(s)$, each $|L_{p}(s)|$ can be bounded above independently
	of $\operatorname{Im}(s)$. Since $|\Gamma (s)|$ has the Schwartz decay
	property on any vertical line, the contribution to the integral of the
	top horizontal segment of $\mathcal{C}_{L}$ tends to zero as
	$\mathcal{C}$ gets larger; the same argument yields the same result for
	the bottom segment.
	
	Overall, taking the limit over a sequence of increasing sparse contours
	$\mathcal{C}$ whose interiors cover the complex plane, the only remaining
	contribution comes from the line $\operatorname{Re}(s) = \sigma $ in
	$\mathcal{C}_{L}$ and we obtain the desired formula
	\begin{equation*}
		\Lambda _{N}(s_{0}) = \frac{1}{2\pi i} \int _{\operatorname{Re}(s) =
			\sigma} \Lambda _{N}^{Euler}(s) \left ( \frac{1}{s-s_{0}} + (-1)^{P}
		\frac{1}{s-k+s_{0}} \right )\ ds.\qedhere
	\end{equation*}
\end{proof}

\begin{remark}
	\label{rmk:approximation-and-ap-asymptotics}
	The bounded growth properties of the $a_{p}$ were crucial in the above
	proof. More precisely, the argument relies on the fact that the line
	$\operatorname{Re}(s) = \sigma $ can be made to lie to the right of all
	the poles of $\Lambda _{N}^{Euler}$, so that the union of the interiors
	of the contours $\mathcal{C}_{L}$ enclose all these poles.
\end{remark}

Combining the lemmas, we obtain an error formula for the approximation.
%
\begin{theorem}[Error formula, integral form]
	\label{thm:error-formula}
	Let $s_{0} \in \mathbb{C}$ be arbitrary. Let
	$\sigma > \max (\operatorname{Re}(s_{0}), \operatorname{Re}(k-s_{0}),
	\frac{k+1}{2})$. Then
	\begin{equation*}
		\Lambda (s_{0}) - \Lambda _{N}(s_{0}) = \int _{\operatorname{Re}(s) =
			\sigma} (\Lambda (s) - \Lambda _{N}^{Euler}(s)) \left (
		\frac{1}{s-s_{0}} + (-1)^{P} \frac{1}{s-k+s_{0}} \right )\ ds.
	\end{equation*}
\end{theorem}
\begin{proof}
	This follows directly from Lemmas \ref{lem:integral-formula-lambda} and \ref{lem:integral-formula-approximation} as the difference between the
	two formulas.
\end{proof}

\begin{remark}
	\label{rem4}
	Similar results hold the derivatives of $\Lambda $ by differentiating under
	the integral sign. Let $s_{0} \in \mathbb{C}$ be arbitrary. For
	$\sigma > \max (\operatorname{Re}(s_{0}), \operatorname{Re}(k-s_{0}))$,
	\begin{equation*}
		\Lambda ^{(n)}(s_{0}) = \frac{n!}{2\pi i} \int _{\operatorname{Re}(s) =
			\sigma} \Lambda (s) \left ( \frac{1}{(s-s_{0})^{n+1}} + (-1)^{n} (-1)^{P}
		\frac{1}{(s-k+s_{0})^{n+1}} \right )\ ds.
	\end{equation*}
	For
	$\sigma > \max (\operatorname{Re}(s_{0}), \operatorname{Re}(k-s_{0}),
	\frac{k+1}{2})$,
	\begin{equation*}
		\Lambda _{N}^{(n)}(s_{0}) = \frac{n!}{2\pi i} \int _{
			\operatorname{Re}(s) = \sigma} \Lambda _{N}^{Euler}(s) \left (
		\frac{1}{(s-s_{0})^{n+1}} + (-1)^{n} (-1)^{P}
		\frac{1}{(s-k+s_{0})^{n+1}} \right )\ ds
	\end{equation*}
	and
	\[\Lambda ^{(n)}(s_{0}) - \Lambda _{N}^{(n)}(s_{0}) = \frac{n!}{2\pi i} \int _{\operatorname{Re}(s) = \sigma} (
	\Lambda (s) - \Lambda _{N}^{Euler}(s)) \left (
	\frac{1}{(s-s_{0})^{n+1}} + (-1)^{n} (-1)^{P}
	\frac{1}{(s-k+s_{0})^{n+1}} \right )\ ds.\]
\end{remark}

\section{Convergence of truncated approximations}
\label{sec:convergence-of-approximations}

Here, we utilize the error formula of Theorem \ref{thm:error-formula} to
prove that the approximations converge to the full completed L-function.
An identical argument, which we omit, ensures convergence of the derivatives
of $\Lambda _{N}$ as well.

\begin{theorem}
	\label{thm:approximations-convergence}
	For any $s_{0} \in \mathbb{C}$,
	$\Lambda _{N}(s_{0}) \to \Lambda (s_{0})$ as $N \to \infty $.
\end{theorem}
\begin{proof}
	Let
	$\sigma > \max (\operatorname{Re}(s_{0}), \operatorname{Re}(k - s_{0}),
	\frac{k+1}{2})$. In particular this means the expansion
	$\Lambda (s) = g(s) \sum _{n=1}^{\infty }a_{n} n^{-s}$ holds on the line
	$\operatorname{Re}(s) = \sigma $. By Theorem \ref{thm:error-formula},
	\begin{equation*}
		\Lambda (s_{0}) - \Lambda _{N}(s_{0}) = \int _{\operatorname{Re}(s) =
			\sigma} (\Lambda (s) - \Lambda _{N}^{Euler}(s)) \left (
		\frac{1}{s-s_{0}} + (-1)^{P} \frac{1}{s-k+s_{0}} \right )\ ds
	\end{equation*}
	and we can expand $\Lambda (s) - \Lambda _{N}^{Euler}(s)$ as
	$g(s) \sum _{n=1}^{\infty }c_{n}^{(N)} n^{-s}$ where
	\begin{equation*}
		c_{n}^{(N)} =
		\begin{cases}
			a_{n} & \text{$n$ has a prime factor $> p_{N}$}
			\\
			0 & \text{$n$ is a product of primes $\leq p_{N}$}.
		\end{cases}
	\end{equation*}
	Thus
	\begin{align*}
		&\left | (\Lambda (s) - \Lambda _{N}^{Euler}(s)) \left (
		\frac{1}{s-s_{0}} + (-1)^{P} \frac{1}{s-k+s_{0}} \right ) \right |
		\\
		&\quad \leq |g(s)| \left ( \sum _{n=1}^{\infty }|c_{n}^{(N)}| n^{-
			\sigma} \right ) \left | \frac{1}{s-s_{0}} + (-1)^{P}
		\frac{1}{s-k+s_{0}} \right |
		\\
		&\quad \leq |g(s)| \left ( \sum _{n=1}^{\infty }|a_{n}| n^{-\sigma}
		\right ) \left | \frac{1}{s-s_{0}} + (-1)^{P} \frac{1}{s-k+s_{0}}
		\right |.
	\end{align*}
	The series $\sum _{n=1}^{\infty }|a_{n}| n^{-\sigma}$ is convergent since
	$\sigma > \frac{k+1}{2}$, and
	$\int _{\operatorname{Re}(s) = \sigma} |g(s)| |
	\frac{1}{s-s_{0}} + (-1)^{P} \frac{1}{s-k+s_{0}} |\ ds$ converges
	thanks to the rapid decay properties of $|g(s)|$, so the integrand is dominated
	by an integrable function independent of $N$. Therefore the hypotheses
	of the dominated convergence theorem are satisfied. It follows that
	\begin{equation*}
		\lim _{N \to \infty} \Lambda (s_{0}) - \Lambda _{N}(s_{0}) = \!\int _{
			\operatorname{Re}(s) = \sigma}\! g(s) \lim _{N \to \infty} \left [
		\sum _{n=1}^{\infty }c_{n}^{(N)} n^{-s} \right ] \left (
		\frac{1}{s-s_{0}} + (-1)^{P} \frac{1}{s-k+s_{0}} \right )\ ds.
	\end{equation*}
	Since the integrand tends to zero pointwise, we obtain the desired conclusion.
\end{proof}

\section{Example: approximating the L-function of the modular discriminant}
\label{sec:example-modular-discriminant}

As an example, we show here a demonstration of the approximation technique
for the L-function of the modular discriminant $\Delta $, whose coefficients
are the well-known Ramanujan tau sequence. This is a modular form of weight
$k = 12$ and level $C = 1$ and its L-function satisfies the functional
equation $\Lambda (s) = \Lambda (12-s)$. Since the generalized Riemann
hypothesis predicts that the zeros of L-functions with positive real part
all lie on the critical line $\operatorname{Re}(s) = 6$, where the Euler
product expansion does not hold, one might use the approximation technique
to investigate the locations of zeros of $\Lambda (s)$ on this line. Fig. \ref{fig:Z-function-plot-modular-discriminant} compares the Z-function
$Z(t) = \frac{\Lambda (6 + it)}{|g(6 + it)|}$ of this L-function and the
approximations $Z_{N}(t) = \frac{\Lambda _{N}(6 + it)}{|g(6 + it)|}$ for
$N = 1, 2, 3$, and Table \ref{tab:Z-function-zeros-modular-discriminant} compares the first eight
zeros of the true Z-function and the approximation $Z_{3}$ using the first
three Euler factors. The data for these computations was sourced from the
LMFDB
\cite[Newform orbit 1.12.a.a, L-function 2-1-1.1-c11-0-0]{lmfdb}.

\begin{figure}
	\includegraphics[width=0.5\linewidth]{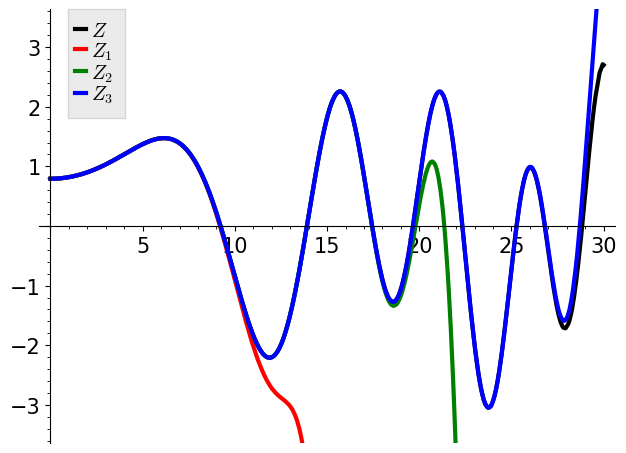}
	\caption{Plot of $Z(t)$ (black), $Z_{1}(t)$ (red), $Z_{2}(t)$ (green), and
		$Z_{3}(t)$ (blue) for $t \in [0, 30]$.}
	\label{fig:Z-function-plot-modular-discriminant}
\end{figure}
%
\begin{table}
	\caption{Comparison of the first eight zeros of $Z(t)$ and $Z_{3}(t)$.}
	\label{tab:Z-function-zeros-modular-discriminant}
	\begin{tabular}{ lll }\hline
		$t_{0} =$ zero of $Z(t)$ & $t_{0}' =$ zero of $Z_{3}(t)$ & Error ($t_{0} - t_{0}'$) \\
		\hline			
		9.2223793999211... & 9.2223793999323... & $-1.11... \cdot 10^{-11}$ \\
		
		13.907549861392... & 13.907549860287... & $1.10... \cdot 10^{-9}$ \\
		
		17.442776978234... & 17.442777058770... & $-8.05... \cdot 10^{-8}$ \\
		
		19.656513141954... & 19.656511952233... & $1.18... \cdot 10^{-6}$ \\
		
		22.336103637209... & 22.336129046421... & $-2.54... \cdot 10^{-5}$ \\
		
		25.274636548112... & 25.273041434242... & $1.59... \cdot 10^{-3}$ \\
		
		26.804391158350... & 26.818461412067... & $-1.40... \cdot 10^{-2}$ \\
		
		28.831682624186... & 28.705434564429... & $1.26... \cdot 10^{-1}$ \\
		\hline \end{tabular}
	%
\end{table}

In Fig. \ref{fig:Z-function-plot-modular-discriminant}, it is visually
apparent that the Z-functions using one, two, and three Euler factors form
successively better approximations to the true Z-function in successively
larger regions around the central point $t = 0$. The data of just three
Euler factors provides a tight approximation to the true Z-function up
to $t = 28$. Table \ref{tab:Z-function-zeros-modular-discriminant} shows
that the approximation using three Euler factors predicts the location
of the first zeros of $Z(t)$ to several decimal places.

\begin{remark}
	\label{rem5}
	For large $t$, $\Lambda (6 + it)$, $\Lambda _{N}(6 + it)$, and
	$g(6 + it)$ are all very small. In order to avoid catastrophic numerical
	error for large $t$, one must use high-precision computations. This can
	be quite memory-intensive. We refer the interested reader to the the supplementary material attached to this article on the Journal of Number Theory website. This supplementary material includes a Sage file used to generate
	Fig. \ref{fig:Z-function-plot-modular-discriminant} and Table \ref{tab:Z-function-zeros-modular-discriminant}, which shows a basic example
	of these high-precision numerical computations.
\end{remark}

\section{Interpretation via Mellin transforms}
\label{sec:mellin-transforms}

In this section, we give another perspective on the approximation procedure
by interpreting it in terms of Mellin transforms. Our starting point is
Corollary \ref{cor:integral-formula-lambda-coefficients}. Interchanging
the sum and the integral, we obtain the identity

\begin{equation*}%
	\small
	\begin{aligned}
		\Lambda (s_{0}) = \sum _{n=1}^{\infty }a_{n} \left (
		\underbrace{\frac{1}{2\pi i} \int _{\operatorname{Re}(s) = \sigma} \frac{\Gamma (s)}{s - s_{0}} \left ( \frac{2 \pi n}{\sqrt{C}} \right )^{-s}\ ds}_{I_{1}}
		+ (-1)^{P}
		\underbrace{\frac{1}{2\pi i} \int _{\operatorname{Re}(s) = \sigma} \frac{\Gamma (s)}{s - k + s_{0}} \left ( \frac{2 \pi n}{\sqrt{C}} \right )^{-s}\ ds}_{I_{2}}
		\right ).
	\end{aligned}
\end{equation*}
We will interpret $I_{1}$ and $I_{2}$ as inverse Mellin transforms. Recall
that the Mellin transform $\mathcal{M}$ and its inverse
$\mathcal{M}^{-1}$ are defined by
\begin{equation*}
	\mathcal{M}\{f\}(s) = \mathcal{M}\{f(x)\}(s) = \int _{0}^{\infty }x^{s-1}
	f(x)\ dx
\end{equation*}
and
\begin{equation*}
	\mathcal{M}^{-1}\{\varphi \}(s) = \mathcal{M}^{-1}\{\varphi (s)\}(x) =
	\frac{1}{2 \pi i} \int _{\operatorname{Re}(s) = \sigma} x^{-s}
	\varphi (s)\ ds
\end{equation*}
for $\sigma $ sufficiently large. There is a convolution-type formula for
Mellin transforms: if
$I(x) = \int _{0}^{\infty }f(t) g(x^{k} t)\ dt$, then
$\mathcal{M}\{I\}(s) = \int _{0}^{\infty }x^{s-1} \int _{0}^{\infty }f(t)
g(x^{k} t)\ dt\ dx$. Substituting $u = x^{k} t$, we get
\begin{equation*}
	\mathcal{M}\{I\}(s) = \frac{1}{k} \int _{0}^{\infty }t^{-s / k} f(t)
	\ dt \int _{0}^{\infty }u^{s / k - 1} g(u)\ du = \frac{1}{k}
	\mathcal{M}\{f\}\left ( 1 - \frac{s}{k} \right ) \mathcal{M}\{g\}
	\left ( \frac{s}{k} \right ).
\end{equation*}
Taking $k = 1$, we get
$M\{\int _{0}^{\infty }f(t) g(xt)\ dt\}(s) = \mathcal{M}\{f\}(1-s)
\mathcal{M}\{g\}(s)$ and so
\begin{equation*}
	\int _{0}^{\infty }f(t) g(xt)\ dt = \mathcal{M}^{-1} {\Big \{}
	\mathcal{M}\{f\}(1-s) \mathcal{M}\{g\}(s) {\Big \}} (x).
\end{equation*}
Observe that $I_{1}$ is exactly the inverse Mellin transform
$\mathcal{M}^{-1}\{\varphi _{1}\}\left ( \frac{2 \pi n}{\sqrt{C}}
\right )$ where
\begin{equation*}
	\varphi _{1}(s) = \frac{\Gamma (s)}{s - s_{0}} = - \left (
	\frac{1}{(1 - s) + (s_{0} - 1)} \right ) \Gamma (s).
\end{equation*}
Since
\begin{equation*}
	\mathcal{M}\{\exp (-x)\}(s) = \Gamma (s) \quad \text{and} \quad
	\mathcal{M}\left \{
	\begin{array}{l @{\quad}l}
		-x^{a} & 1 < x < \infty
		\\
		0 & 0 < x < 1
	\end{array}
	\right \}(s) = \frac{1}{s+a},
\end{equation*}
the latter being valid for
$\operatorname{Re}(s) < -\operatorname{Re}(a)$, we see that
$\mathcal{M}^{-1}\{\varphi _{1}\}(x) = \int _{0}^{\infty }f(t) g(xt)
\ dt = \int _{1}^{\infty }t^{s_{0}-1} \exp (-xt)\ dt$, and so
\begin{equation*}
	I_{1} = \mathcal{M}^{-1}\{\varphi _{1}\}\left (
	\frac{2 \pi n}{\sqrt{C}} \right ) = \int _{1}^{\infty }t^{s_{0} - 1}
	\exp \left ( \frac{-2\pi n t}{\sqrt{C}} \right ) \ dt,
\end{equation*}
which converges for all $s_{0} \in \mathbb{C}$. Similarly
$I_{2} = \mathcal{M}^{-1}\{\varphi _{2}\}\left (
\frac{2 \pi n}{\sqrt{C}} \right )$ where
$\varphi _{2}(s) = \frac{\Gamma (s)}{s - k + s_{0}}$, and an analogous
argument shows that
\begin{equation*}
	I_{2} = \int _{1}^{\infty }t^{k-1-s_{0}} \exp \left (
	\frac{-2\pi n t}{\sqrt{C}} \right )\ dt.
\end{equation*}
Putting these two pieces together, we see that
\begin{align*}
	\Lambda (s_{0}) &= \sum _{n=1}^{\infty }a_{n} \left ( \int _{1}^{
		\infty }t^{s_{0} - 1} \exp \left ( \frac{-2\pi n t}{\sqrt{C}} \right )
	\ dt + (-1)^{P} \int _{1}^{\infty }t^{k-1-s_{0}} \exp \left (
	\frac{-2\pi n t}{\sqrt{C}} \right )\ dt \right )
	\\
	&= \int _{1}^{\infty }(t^{s_{0}-1} + (-1)^{P} t^{k-1-s_{0}}) \sum _{n=1}^{
		\infty }a_{n} \exp \left ( \frac{-2\pi n t}{\sqrt{C}} \right )\ dt.
\end{align*}
Now, owing to the fact that the $a_{n}$ are the coefficients of the modular
form $f$, i.e. $f(z) = \sum _{n=1}^{\infty }a_{n} \exp (2\pi i n z)$, we
obtain the following well-known result (see e.g.
\cite[pp. 139-141]{koblitz}).
%
\begin{proposition}
	\label{prop:mellin-transform-Lambda}
	For all $s \in \mathbb{C}$,
	\begin{equation*}
		\Lambda (s) = \int _{1}^{\infty }(t^{s-1} + (-1)^{P} t^{k-1-s}) f
		\left ( \frac{it}{\sqrt{C}} \right )\ dt.
	\end{equation*}
\end{proposition}
Lemma \ref{lem:integral-formula-approximation} exhibits a corresponding
contour integral representation of the approximations $\Lambda _{N}$ via
the truncated Euler product $\Lambda _{N}^{Euler}$, which corresponds to
taking a subseries of the $a_{n}$. Defining
$f_{N}(z) = \sum _{n=1}^{\infty }b_{n}^{(N)} \exp (2\pi i n z)$, where
\begin{equation*}
	b_{n}^{(N)} = a_{n} - c_{n}^{(N)} =
	\begin{cases}
		0 & \text{$n$ has a prime factor $> p_{N}$}
		\\
		a_{n} & \text{$n$ is a product of primes $\leq p_{N}$},
	\end{cases}
\end{equation*}
an analogous manipulation then shows that $\Lambda _{N}$ is then given
by a similar convolution-type formula, now with $f_{N}$.
%
\begin{proposition}
	\label{prop:mellin-transform-approximation}
	For all $s \in \mathbb{C}$,
	\begin{equation*}
		\Lambda _{N}(s) = \int _{1}^{\infty }(t^{s-1} + (-1)^{P} t^{k-1-s}) f_{N}
		\left ( \frac{it}{\sqrt{C}} \right )\ dt.
	\end{equation*}
\end{proposition}
Proposition \ref{prop:mellin-transform-approximation} can be reinterpreted
as an alternative explicit definition of $\Lambda _{N}$ valid on all of
$\mathbb{C}$, and gives another perspective on the construction outlined
in Section~\ref{sec:approximation-construction}. Furthermore, combining
Propositions \ref{prop:mellin-transform-Lambda} and \ref{prop:mellin-transform-approximation}, we see that
\begin{equation*}
	\Lambda (s) - \Lambda _{N}(s) = \int _{1}^{\infty }(t^{s_{0}-1} + (-1)^{P}
	t^{k-1-s_{0}}) \sum _{n=1}^{\infty }c_{n}^{(N)} \exp \left (
	\frac{-2\pi n t}{\sqrt{C}} \right )\ dt.
\end{equation*}
Owing to the identity
$\int _{1}^{\infty }t^{z-1} \exp (-at)\ dt = a^{-z} \Gamma (z, a)$ where
$\Gamma (z, a) = \int _{a}^{\infty }t^{z-1} \exp (-t)\ dt$ is the incomplete
gamma function, we may restate the above equation into the following version
of the approximation error formula.
%
\begin{corollary}[Error formula, series form]
	\label{cor:error-bound-incomplete-gamma-function}
	\begin{equation*}
		\Lambda (s) - \Lambda _{N}(s) = \sum _{n=1}^{\infty} c_{n}^{(N)}
		\left ( \left ( \frac{2\pi n}{\sqrt{C}} \right )^{-s} \Gamma \left ( s,
		\frac{2\pi n}{\sqrt{C}} \right ) + (-1)^{P} \left (
		\frac{2\pi n}{\sqrt{C}} \right )^{-k+s} \Gamma \left ( k-s,
		\frac{2\pi n}{\sqrt{C}} \right ) \right )
	\end{equation*}
\end{corollary}
We note that Corollary \ref{cor:error-bound-incomplete-gamma-function} is essentially the analog
of Theorem 1 in \cite{alzergani2023family}, Theorem 5 in
\cite{nastasescu2023visual}, and Lemma 3.4 in
\cite{nastasescu2023riemann}. Interpreting the approximation through the
lens of Mellin transforms and via the convolution integral definition may
help to explain the appearance of formulas of this type in the context
of L-function approximations.

Corollary \ref{cor:error-bound-incomplete-gamma-function} implies that
the decay of the error of the approximation is at least exponential in
$N$, in the following sense.
%
\begin{proposition}
	\label{prop:error-first-term-exponential-decay}
	Fix $s \in \mathbb{C}$. For large enough $N$, the first nonzero term
	$T(N)$ in the series expansion of $\Lambda (s) - \Lambda _{N}(s)$ decays
	superexponentially as a function of $N$. The same is true for the first
	nonzero term in the series expansion of
	$\Lambda ^{(n)}(s) - \Lambda _{N}^{(n)}(s)$.
\end{proposition}
\begin{proof}
	The first nonzero term $T(N)$ of the series is the $n = p_{N+1}$ term.
	Writing $\sigma = \operatorname{Re}(s)$, by Theorem \ref{thm:petersson-ramanujan}, $T$ is bounded by
	\begin{equation*}
		\begin{aligned}
			|T(N)| &\leq 2 p_{N+1}^{(k-1)/2}
			\\
			&\quad {}\times\left ( \left (
			\frac{2\pi p_{N+1}}{\sqrt{C}} \right )^{-\sigma} \left | \Gamma
			\left ( s, \frac{2\pi p_{N+1}}{\sqrt{C}} \right ) \right | + \left (
			\frac{2\pi p_{N+1}}{\sqrt{C}} \right )^{\sigma - k} \left | \Gamma
			\left ( k - s, \frac{2\pi p_{N+1}}{\sqrt{C}} \right ) \right |
			\right ).
		\end{aligned}
	\end{equation*}
	For any $\sigma \in \mathbb{R}$, there is a positive real number
	$m(\sigma )$ so that $t^{\sigma - 1} < e^{t/2}$ for all
	$t > m(\sigma )$. Thus, whenever $a > m$,
	\begin{equation*}
		|\Gamma (s, a)| \leq \int _{a}^{\infty }t^{\sigma - 1} e^{-t}\ dt
		\leq \int _{a}^{\infty }e^{-t/2}\ dt = 2e^{-a/2}
	\end{equation*}
	and similarly for $|\Gamma (k-s, a)|$. It follows that whenever $N$ is
	large enough so that
	$\frac{2\pi p_{N+1}}{\sqrt{C}} > \max (m(\sigma ), m(k-\sigma ))$,
	\begin{equation*}
		|T(N)| \leq 4 e^{-\pi p_{N+1} / \sqrt{C}} p_{N+1}^{(k-1)/2} \left (
		\left ( \frac{2\pi p_{N+1}}{\sqrt{C}} \right )^{-\sigma} + \left (
		\frac{2\pi p_{N+1}}{\sqrt{C}} \right )^{\sigma - k} \right ).
	\end{equation*}
	Since $n \log n < p_{n} < n (\log n + \log \log n)$ for $n \geq 6$ by a
	theorem of Rosser \cite{rosser}, for $N \geq 5$ we have
	\begin{align*}
		|T(N)| &< 4 (N+1)^{-\pi (N+1) / \sqrt{C}} ((N+1) (\log (N+1) + \log
		\log (N+1)))^{(k-1)/2}
		\\
		&\qquad \times \left ( \left ( \frac{2\pi (N+1) \log (N+1)}{\sqrt{C}}
		\right )^{-\sigma} + \left ( \frac{2\pi (N+1) \log (N+1)}{\sqrt{C}}
		\right )^{\sigma - k} \right )
	\end{align*}
	for $s$ in the strip $0 < \operatorname{Re}(s) < k$. A similar bound holds
	for $s$ to the left or right of this strip by straightforwardly modifying
	the terms in the second line of the equation, yielding the first claim.
	The argument for $\Lambda ^{(n)}(s) - \Lambda _{N}^{(n)}(s)$ is similar:
	in this case $|T(N)|$ can be bounded by $2p_{N+1}^{(k-1)/2}$ times a sum
	of products of $s$-derivatives of
	$\left ( \frac{2\pi p_{N+1}}{\sqrt{C}} \right )^{-s}$,
	$\left ( \frac{2\pi p_{N+1}}{\sqrt{C}} \right )^{s-k}$,
	$\Gamma \left ( s, \frac{2\pi p_{N+1}}{\sqrt{C}} \right )$, and
	$\Gamma \left ( k-s, \frac{2\pi p_{N+1}}{\sqrt{C}} \right )$, so it suffices
	to bound the derivatives of $\Gamma (s, a)$. Since
	$\frac{d^{n}}{ds^{n}} \Gamma (s, a) = \int _{a}^{\infty }t^{s-1} (
	\log t)^{n} e^{-t}\ dt$, this derivative is bounded by $2 e^{-a/2}$ whenever
	$a > \max (1, m(\sigma + n))$, and the second statement of the proposition
	follows.
\end{proof}
The proof of Proposition \ref{prop:error-first-term-exponential-decay} can be straightforwardly
adapted -- analyzing each term of the series using the exponential bound
on $\Gamma (s, a)$ and its derivatives -- to show that, for any fixed
$s \in \mathbb{C}$, the overall error
$|\Lambda (s) - \Lambda _{N}(s)|$ or
$|\Lambda ^{(n)}(s) - \Lambda _{N}^{(n)}(s)|$ is bounded by an at least
exponentially decreasing function of $N$ for large enough $N$.

%

\section*{Acknowledgements}
We thank Yuri Matiyasevich for bringing to our attention an error in the proof of Lemma \ref{lem11}. We sincerely thank the anonymous referee for their valuable feedback on
the manuscript. We thank Bogdan Stoica for his helpful comments on an earlier
version of the paper. The second author also thanks Tariq Osman for conversations
which inspired the results in Section~\ref{sec:mellin-transforms} of this
paper. A.H. was supported by the Simons Collaboration for Mathematicians
under Award Number 708790.
K.S. was supported by the 2024 Brandeis University
Harold I. Levine Endowed Fellowship.

\appendix

\section{Detailed proofs of results in Section \ref{sec:approximation-construction}}
\label{app:detailed-arguments}

Here, we provide more detailed proofs of the results in Section~\ref{sec:approximation-construction}. First we focus on proving the convergence
of the sum defining $\Lambda _{N}^{pp}(s)$.

\begin{lemma}
\label{lem:pp-Lpn-convergent}
The function
$\Lambda _{N, n}^{pp}(s) = \sum _{s_{\star }
	\text{ pole of $L_{p_{n}}(s)$}} \sum _{k = -k_{s_{\star}}}^{-1} \rho _{s_{
		\star}}^{(k)} (s-s_{\star})^{k}$ is convergent for $s$ not a pole of
$L_{p_{n}}(s)$.
\end{lemma}
\begin{proof}
For $s_{\star}$ a pole of $L_{p}(s)$, define
$\Lambda _{N}^{pp, s_{\star}}$ to be the principal part of the Laurent
expansion of $\Lambda _{N}$ at $s = s_{\star}$. Taking a Laurent expansion
of each factor, we have
\begin{align*}
	\Lambda _{N}^{Euler}(s) &= g(s) \prod _{p=2}^{p_{N}} L_{p}(s) = g(s)
	\prod _{p=2}^{p_{N}} \frac{1}{1 - \alpha _{1,p} p^{-s}}
	\frac{1}{1 - \alpha _{2,p} p^{-s}}
	\\
	&= \left ( \sum _{k=0}^{\infty }\rho _{g, s_{\star}}^{(k)} (s - s_{
		\star})^{k} \right ) \prod _{p=2}^{p_{N}} \left [ \left ( \sum _{k=-1}^{
		\infty }\rho _{p_{n}, 1, s_{\star}}^{(k)} (s - s_{\star})^{k} \right ) \left ( \sum _{k=-1}^{\infty }\rho _{p_{n}, 2, s_{\star}}^{(k)} (s - s_{
		\star})^{k} \right ) \right ]
	\\
	&= \sum _{k=-2N}^{\infty }\sum _{
		\substack{A \in \mathbb{Z}_{\geq -1}^{2N+1} \\ |A| = k}} \rho _{s_{
			\star}}^{(A)} (s - s_{\star})^{k}
\end{align*}
where
$\rho _{s_{\star}}^{(A)} = \rho _{p_{1}, 1, s_{\star}}^{(a_{1})}
\cdots \rho _{p_{N}, 1, s_{\star}}^{(a_{N})} \rho _{p_{1}, 2, s_{
		\star}}^{(a_{N+1})} \cdots \rho _{p_{N}, 2, s_{\star}}^{(a_{2N})}
\rho _{g, s_{\star}}^{(a_{2N+1})}$. Hence the principal part at
$s_{\star}$ is
$\Lambda _{N}^{pp, s_{\star}}(s) = \sum _{k=-2N}^{-1} \sum _{
	\substack{A \in \mathbb{Z}_{\geq -1}^{2N+1} \\ |A| = k}} \rho _{s_{
		\star}}^{(A)} (s - s_{\star})^{k}$ and so
\begin{equation*}
	\Lambda _{N,n}^{pp}(s) = \sum _{k = -2N}^{-1} \sum _{s_{\star }
		\text{ pole of $L_{p_{n}}(s)$}} \rho _{s_{\star}}^{(A)} (s - s_{\star})^{k}.
\end{equation*}
We set
\begin{equation*}
	M = \max _{
		\substack{1 \leq i \leq N \\ j = 1 \text{ or } 2 \\ -1 \leq k \leq 2N-1 \\ s_{\star }\text{pole of $L_{p_{n}}$}}}
	\left | \frac{\rho _{p_{i}, j, s_{\star}}^{(k)}}{s_{\star}^{2k+2}}
	\right |,
\end{equation*}
which is finite by Lemma \ref{lem:laurent-coefficients-bounded-growth}, and in plain language is
the maximum over all the possible Laurent coefficients of the
$L_{p_{n}}$ factors which could appear in $\rho _{s_{\star}}^{(A)}$, normalized
by an appropriate power of $|s_{\star}|$. This gives the crude bound
\begin{equation*}
	|\Lambda _{N,n}^{pp}(s)| \leq \sum _{k = -2N}^{-1} \sum _{\ell = 0}^{2N-1}
	c_{\ell }M^{2N} \sum _{s_{\star }\text{ pole of $L_{p_{n}}(s)$}} |
	\rho _{g, s_{\star}}^{(\ell )}| (1 + |s_{\star}|)^{p_{\ell}} |s - s_{
		\star}|^{k}
\end{equation*}
where $c_{\ell}$ is some combinatorial constant counting the number of
multiindices $A$ whose last entry is $\ell $ and $p_{\ell}$ is the maximum
of $\sum _{i=1}^{2N} (2a_{i}+2)$ over such multiindices $A$. The main control
will come from the $\rho _{g, s_{\star}}^{(a_{2N+1})}$, and each term in
the expansion of this coefficient will contain some derivative of
$\Gamma (s)$.

For any $a>0$, the function $\Gamma (a + iy)$ is the Fourier transform
of a Schwartz function, hence Schwartz itself. Via the functional equation
$z\Gamma (z) = \Gamma (z+1)$, it follows that $\Gamma $ has the Schwartz
decay property on any vertical line in the complex plane, and so does each
derivative of $\Gamma $. Hence, the
$\rho _{g, s_{\star}}^{(b_{N+1})}$ in each term of
$\Lambda _{N}^{pp, s_{\star}}(s)$ decays rapidly as
$\operatorname{Im}s_{\star }\to \pm \infty $, ensuring
$\Lambda _{N, n}^{pp}(s) = \sum _{s_{\star }
	\text{ pole of $L_{p_{n}}(s)$}} \Lambda _{N}^{pp, s_{\star}}(s)$ converges,
because the poles of each single $L_{p}(s)$ form two (possibly coinciding)
sets on the line $\operatorname{Re}(s) = \frac{k-1}{2}$ spaced at uniform
distance.
\end{proof}

\begin{lemma}
\label{lem:pp-g-convergent}
The function
$\Lambda _{N,g}^{pp}(s) = \sum _{s_{\star }\text{ pole of g(s)}} \sum _{k
	= -k_{s_{\star}}}^{-1} \rho _{s_{\star}}^{(k)} (s - s_{\star})^{k}$ is
convergent for $s$ not a pole of $g(s)$.
\end{lemma}
\begin{proof}
Since $\Gamma $ has simple poles at the nonpositive integers, so does
$\Lambda _{N}^{Euler}$, and so
\begin{align*}
	\sum _{s_{\star }\in \mathbb{Z}_{\leq 0}} \Lambda _{N}^{pp, s_{\star}}(s)
	&= \sum _{-n \in \mathbb{Z}_{\leq 0}} \operatorname{Res}_{s=-n}(
	\Lambda _{N}^{Euler}(s)) (s+n)^{-1}
	\\
	&= \sum _{-n \in \mathbb{Z}_{\leq 0}} (\operatorname{Res}_{s=-n} g(s))
	\prod _{p=2}^{p_{N}} L_{p}(-n) (s+n)^{-1}
	\\
	&= \sum _{-n \in \mathbb{Z}_{\leq 0}} C^{-n/2} (2\pi )^{n} (
	\operatorname{Res}_{s=-n} \Gamma (s)) \prod _{p=2}^{p_{N}} L_{p}(-n) (s+n)^{-1}
	\\
	&= \sum _{-n \in \mathbb{Z}_{\leq 0}}
	\frac{(-1)^{n} (2\pi )^{n} \prod _{p=2}^{p_{N}} L_{p}(-n)}{C^{n/2} n!}
	(s+n)^{-1}.
\end{align*}
Since each $|L_{p}(-n)|$ is eventually decreasing as $n \to \infty $, their
product can be bounded uniformly for all $n$, and therefore the $n!$ term
in the denominator ensures convergence.
\end{proof}

With these tools in hand, Proposition \ref{prop:lambda-pp-convergent-sparse-contour-bound} proves the convergence
of $\Lambda _{N}^{pp}(s)$ whenever $s$ is not a pole of
$\Lambda _{N}^{Euler}$. Next we turn our attention to proving in detail
the bound of $\Lambda _{N}^{pp}(s)$ on sparse contours.

\begin{definition}
\label{defn2}
Let $N$ be fixed, and let $a, b > 0$ be constants depending only on
$N$. We say that a closed contour $\mathcal{C}$ in $\mathbb{C}$ is
$(a,b)$-sparse with respect to $S_{N}$ if for all
$s \in \mathcal{C}$,
\begin{equation*}
	\min _{s_{\star }\in D_{N}} |s - s_{\star}| \geq a \quad \text{and}
	\quad \min _{s_{\star }\in \mathbb{Z}_{\leq 0}} |s - s_{\star}| \geq b.
\end{equation*}
\end{definition}
%
\begin{lemma}
\label{lem11}
For some $a, b > 0$ depending only on $N$, there exist arbitrarily large
rectangular contours $\mathcal{C}$ which are $(a,b)$-sparse with respect
to $S_{N}$.
\end{lemma}
\begin{proof}
If the sides of the rectangle are defined by the lines
$\operatorname{Re}(s) = \sigma _{1} < 0$,
$\operatorname{Re}(s) = \sigma _{2} > 0$,
$\operatorname{Im}(s) = \tau _{1} < 0$, and
$\operatorname{Im}(s) = \tau _{2} > 0$, then it suffices to show that we
can make $|\sigma _{1}|, |\sigma _{2}|, |\tau _{1}|, |\tau _{2}|$ arbitrarily
large for some fixed $a$ and $b$. Picking
$\sigma _{1} = \ell - \frac{1}{2}$ for
$\ell \in \mathbb{Z}_{\leq 0}$ and $\sigma _{2} > 1$ arbitrarily large
clearly achieves the second condition of a sparse contour with
$b = \frac{1}{2}$. For the first condition, we proceed as follows. Consider
an arbitrary interval $I$ of length $1$, and let
$D_{I} = \{s_{\star }\in D_{N} \ \big |\ \operatorname{Im}(s_{\star})
\in I\}$. Since the poles of a single $L_{p}$ comprise two (possibly coinciding)
sets on the vertical line $\operatorname{Re}(s) = \frac{k-1}{2}$, each
spaced by a distance of $\frac{2\pi}{\log p}$ apart on the line, it follows
that each $L_{p}$ contributes at most
$2 \frac{\log p}{2 \pi} + 2 \leq \frac{\log p_{N}}{\pi} + 2$ elements to
$D_{I}$, hence $\#(D_{I}) \leq \frac{N \log p_{N}}{\pi} + 2N$. The average length of the subintervals of $I$ marked by the elements of $D_I$ is thus
at least $\frac{\pi}{N \log(p_N) + (2N + 1) \pi}$, so there is $\tau _{1} \in I$ such
that the line $\operatorname{Im}(s) = \tau _{1}$ remains at least distance
$\frac{\pi}{2N \log(p_N) + (4N+2) \pi}$ away from all elements of $D_{N}$. The same
is true for $\tau _{2}$, so $a = \frac{\pi}{2N \log(p_N) + (4N+2) \pi}$ satisfies
the required condition.
\end{proof}

Hereafter, we suppress the constants $a$ and $b$ in the notation and refer
to $(a,b)$-sparse contours simply as sparse contours. To control the magnitude
of $\Lambda _{N}^{pp}$ on sparse contours, we start with a quick technical
lemma.
%
\begin{lemma}
\label{lem:technical-bounding-lemma}
Let $k \geq 1$ and suppose that $|s - s_{\star}| \geq c$ and
$|s_{\star}| \geq d$. Then
\begin{equation*}
	\frac{1 + |s|}{|s_{\star}|^{k} |s - s_{\star}|^{k}} \leq
	\frac{1}{c^{k} d^{k}} + \left ( \frac{1}{c} + \frac{1}{d} \right )
	\frac{1}{c^{k-1} d^{k-1}}.
\end{equation*}
\end{lemma}
\begin{proof}
\begin{align*}
	\frac{1 + |s|}{|s_{\star}|^{k} |s - s_{\star}|^{k}} &=
	\frac{1}{|s_{\star}|^{k} |s - s_{\star}|^{k}} +
	\frac{|s|}{|s_{\star}|^{k} |s - s_{\star}|^{k}}
	\\
	&\leq \frac{1}{c^{k} d^{k}} + \frac{|s|}{|s_{\star}| |s - s_{\star}|}
	\frac{1}{|s_{\star}|^{k-1} |s - s_{\star}|^{k-1}}
	\\
	&= \frac{1}{c^{k} d^{k}} + \left | \frac{1}{s - s_{\star}} +
	\frac{1}{s_{\star}} \right |
	\frac{1}{|s_{\star}|^{k-1} |s - s_{\star}|^{k-1}}
	\\
	&\leq \frac{1}{c^{k} d^{k}} + \left ( \frac{1}{c} + \frac{1}{d}
	\right ) \frac{1}{c^{k-1} d^{k-1}}.\qedhere
\end{align*}
\end{proof}

With this technical lemma in hand, we can now produce bounds for
$\Lambda _{N,g}$ and each $\Lambda _{N, n}$ on sparse contours.
%
\begin{lemma}
\label{lem:pp-g-sparse-contour-bound}
For $s \in \mathcal{C}$ a rectangular sparse contour,
$|\Lambda _{N,g}^{pp}(s)| \leq \frac{K_{g}}{1+|s|}$ where $K_{g}$ is a
constant depending only on the $L$-function and $N$.
\end{lemma}
\begin{proof}
It is equivalent to show that $(1+|s|) |\Lambda _{N,g}^{pp}(s)|$ is bounded
above by a constant. We know from a proof of an earlier lemma that
\begin{equation*}
	\Lambda _{N,g}^{pp}(s) = \sum _{n \geq 0}
	\frac{(-1)^{n} (2\pi )^{n}}{C^{n/2} n!} \prod _{p=2}^{p_{N}} L_{p}(-n)
	(s+n)^{-1}.
\end{equation*}
Since all the $L_{p}(-n)$'s are eventually decreasing, we may set
\begin{equation*}
	M = \max _{n \in \mathbb{Z}_{\geq 0}} \prod _{p=2}^{p_{N}} |L_{p}(-n)|
\end{equation*}
to obtain
\begin{equation*}
	|\Lambda _{N,g}^{pp}(s)| \leq M \sum _{n \geq 0} \left (
	\frac{2\pi}{\sqrt{C}} \right )^{n} \frac{1}{n!} \frac{1}{|s+n|},
\end{equation*}
and therefore
\begin{equation*}
	(1+|s|) |\Lambda _{N,g}^{pp}(s)| \leq M \sum _{n \geq 0} \left (
	\frac{2\pi}{\sqrt{C}} \right )^{n} \frac{1}{(n-1)!}
	\frac{1+|s|}{n |s+n|}
\end{equation*}
where we make the convention that
$\frac{1}{(n-1)!} \frac{1+|s|}{n|s+n|} = \frac{1+|s|}{|s|}$ if $n=0$. For
$n \geq 1$, we apply Lemma \ref{lem:technical-bounding-lemma} with
$k = 1$ and $s_{\star }= -n$. By definition of a sparse contour, the constants
appearing in the hypotheses of the Lemma are independent of $s$ and
$n$, so $\frac{1+|s|}{n|s+n|}$ is bounded by a constant independent of
$s$ and $n$. Therefore the $(n-1)!$ term ensures convergence of the sum,
and the result is that $(1+|s|) |\Lambda _{N,g}^{pp}(s)|$ is bounded by
a constant independent of $s$, as was required.
\end{proof}

\begin{lemma}
\label{lem:pp-Lpn-sparse-contour-bound}
For $s \in \mathcal{C}$ a rectangular sparse contour,
$|\Lambda _{N,n}^{pp}(s)| \leq \frac{K_{n}}{1+|s|}$ where $K_{n}$ is a
constant depending only on the $L$-function, $N$, and $n$.
\end{lemma}
\begin{proof}
From the proof of Lemma \ref{lem:pp-Lpn-convergent}, we have the bound
\begin{equation*}
	|\Lambda _{N,n}^{pp}(s)| \leq \sum _{k = 1}^{2N} \sum _{\ell = 0}^{2N-1}
	c_{\ell }M^{2N} S_{k, \ell}(s)
\end{equation*}
where
$S_{k, \ell}(s) = \sum _{s_{\star }\text{ pole of $L_{p_{n}}(s)$}} |
\rho _{g, s_{\star}}^{(\ell )}| (1 + |s_{\star}|)^{p_{\ell}} |s - s_{
	\star}|^{-k}$. It thus suffices to prove that
$(1+|s|) S_{k,\ell}(s)$ is bounded by a constant depending only on
$k$ and $\ell $. Consider
\begin{align*}
	(1 + |s|) S_{k, \ell}(s) &= \sum _{s_{\star }
		\text{ pole of $L_{p_{n}}(s)$}} |\rho _{g, s_{\star}}^{(\ell )}| (1 + |s_{
		\star}|)^{p_{\ell}} \frac{1 + |s|}{|s - s_{\star}|^{k}}
	\\
	&= \sum _{s_{\star }\text{ pole of $L_{p_{n}}(s)$}} |\rho _{g, s_{
			\star}}^{(\ell )}| (1 + |s_{\star}|)^{p_{\ell}} |s_{\star}|^{k}
	\frac{1 + |s|}{|s_{\star}|^{k} |s - s_{\star}|^{k}}.
\end{align*}
Recall that the $s_{\star}$ comprise one set of double poles or two sets
of simple poles on the line $\operatorname{Re}(s) = \frac{k-1}{2}$. Since
$g$ decays rapidly on the vertical line
$\operatorname{Re}(s) = \frac{k-1}{2}$, this means that the sum
\begin{equation*}
	\sum _{s_{\star }\text{ pole of $L_{p_{n}}(s)$}} |\rho _{g, s_{\star}}^{(
		\ell )}| (1 + |s_{\star}|)^{p_{\ell}} |s_{\star}|^{k}
\end{equation*}
converges and its value depends only on $k$ and $\ell $, and so it suffices
to prove that $\frac{1 + |s|}{|s_{\star}|^{k} |s - s_{\star}|^{k}}$ is
bounded above by a constant independent of $s$ and $s_{\star}$. By definition
of a sparse contour $|s - s_{\star}|$ is bounded below by a constant independent
of $s$ and $s_{\star}$, and since every $s_{\star}$ lies on the line
$\operatorname{Re}(s) = \frac{k-1}{2}$, we have
$|s_{\star}| \geq \frac{k-1}{2}$ independently of $s_{\star}$. So applying
Lemma \ref{lem:technical-bounding-lemma}, the constants appearing in the
hypotheses of the Lemma are independent of $s$ and $s_{\star}$, and therefore
the statement follows.
\end{proof}
Once again, with these bounds in hand on each piece, Proposition \ref{prop:lambda-pp-convergent-sparse-contour-bound} now proves that
$\Lambda _{N}^{pp}(s)$ is bounded by $\frac{K}{1 + |s|}$ for some constant
$K$ depending only on the L-function and $N$, whenever $s$ belongs to a
sparse contour.

\printbibliography
	
\end{document}